\documentclass[12pt]{amsart}

\usepackage{amssymb}
\usepackage{amsmath, amsfonts}

\usepackage[mathscr]{eucal}

\input xypic
\xyoption{all}

\makeindex
\makeglossary

\begin{document}

\baselineskip = 16pt

\newcommand \ZZ {{\mathbb Z}}
\newcommand \NN {{\mathbb N}}
\newcommand \RR {{\mathbb R}}
\newcommand \PR {{\mathbb P}}
\newcommand \AF {{\mathbb A}}
\newcommand \GG {{\mathbb G}}
\newcommand \QQ {{\mathbb Q}}
\newcommand \CC {{\mathbb C}}
\newcommand \bcA {{\mathscr A}}
\newcommand \bcC {{\mathscr C}}
\newcommand \bcD {{\mathscr D}}
\newcommand \bcF {{\mathscr F}}
\newcommand \bcG {{\mathscr G}}
\newcommand \bcH {{\mathscr H}}
\newcommand \bcM {{\mathscr M}}
\newcommand \bcI {{\mathscr I}}
\newcommand \bcJ {{\mathscr J}}
\newcommand \bcK {{\mathscr K}}
\newcommand \bcL {{\mathscr L}}
\newcommand \bcO {{\mathscr O}}
\newcommand \bcP {{\mathscr P}}
\newcommand \bcQ {{\mathscr Q}}
\newcommand \bcR {{\mathscr R}}
\newcommand \bcS {{\mathscr S}}
\newcommand \bcT {{\mathscr T}}
\newcommand \bcV {{\mathscr V}}
\newcommand \bcU {{\mathscr U}}
\newcommand \bcW {{\mathscr W}}
\newcommand \bcX {{\mathscr X}}
\newcommand \bcY {{\mathscr Y}}
\newcommand \bcZ {{\mathscr Z}}
\newcommand \goa {{\mathfrak a}}
\newcommand \gob {{\mathfrak b}}
\newcommand \goc {{\mathfrak c}}
\newcommand \gom {{\mathfrak m}}
\newcommand \gon {{\mathfrak n}}
\newcommand \gop {{\mathfrak p}}
\newcommand \goq {{\mathfrak q}}
\newcommand \goQ {{\mathfrak Q}}
\newcommand \goP {{\mathfrak P}}
\newcommand \goM {{\mathfrak M}}
\newcommand \goN {{\mathfrak N}}
\newcommand \uno {{\mathbbm 1}}
\newcommand \Le {{\mathbbm L}}
\newcommand \Spec {{\rm {Spec}}}
\newcommand \Gr {{\rm {Gr}}}
\newcommand \Pic {{\rm {Pic}}}
\newcommand \Jac {{{J}}}
\newcommand \Alb {{\rm {Alb}}}
\newcommand \alb {{\rm {alb}}}
\newcommand \Corr {{Corr}}
\newcommand \Chow {{\mathscr C}}
\newcommand \Sym {{\rm {Sym}}}
\newcommand \Prym {{\rm {Prym}}}
\newcommand \cha {{\rm {char}}}
\newcommand \eff {{\rm {eff}}}
\newcommand \tr {{\rm {tr}}}
\newcommand \Tr {{\rm {Tr}}}
\newcommand \pr {{\rm {pr}}}
\newcommand \ev {{\it {ev}}}
\newcommand \cl {{\rm {cl}}}
\newcommand \interior {{\rm {Int}}}
\newcommand \sep {{\rm {sep}}}
\newcommand \td {{\rm {tdeg}}}
\newcommand \alg {{\rm {alg}}}
\newcommand \im {{\rm im}}
\newcommand \gr {{\rm {gr}}}
\newcommand \op {{\rm op}}
\newcommand \Hom {{\rm Hom}}
\newcommand \Hilb {{\rm Hilb}}
\newcommand \Sch {{\mathscr S\! }{\it ch}}
\newcommand \cHilb {{\mathscr H\! }{\it ilb}}
\newcommand \cHom {{\mathscr H\! }{\it om}}
\newcommand \colim {{{\rm colim}\, }}
\newcommand \End {{\rm {End}}}
\newcommand \coker {{\rm {coker}}}
\newcommand \id {{\rm {id}}}
\newcommand \van {{\rm {van}}}
\newcommand \spc {{\rm {sp}}}
\newcommand \Ob {{\rm Ob}}
\newcommand \Aut {{\rm Aut}}
\newcommand \cor {{\rm {cor}}}
\newcommand \Cor {{\it {Corr}}}
\newcommand \res {{\rm {res}}}
\newcommand \red {{\rm{red}}}
\newcommand \Gal {{\rm {Gal}}}
\newcommand \PGL {{\rm {PGL}}}
\newcommand \Bl {{\rm {Bl}}}
\newcommand \Sing {{\rm {Sing}}}
\newcommand \spn {{\rm {span}}}
\newcommand \Nm {{\rm {Nm}}}
\newcommand \inv {{\rm {inv}}}
\newcommand \codim {{\rm {codim}}}
\newcommand \Div{{\rm{Div}}}
\newcommand \CH{{\rm{CH}}}
\newcommand \sg {{\Sigma }}
\newcommand \DM {{\sf DM}}
\newcommand \Gm {{{\mathbb G}_{\rm m}}}
\newcommand \tame {\rm {tame }}
\newcommand \znak {{\natural }}
\newcommand \lra {\longrightarrow}
\newcommand \hra {\hookrightarrow}
\newcommand \rra {\rightrightarrows}
\newcommand \ord {{\rm {ord}}}
\newcommand \Rat {{\mathscr Rat}}
\newcommand \rd {{\rm {red}}}
\newcommand \bSpec {{\bf {Spec}}}
\newcommand \Proj {{\rm {Proj}}}
\newcommand \pdiv {{\rm {div}}}
\newcommand \wt {\widetilde }
\newcommand \ac {\acute }
\newcommand \ch {\check }
\newcommand \ol {\overline }
\newcommand \Th {\Theta}
\newcommand \cAb {{\mathscr A\! }{\it b}}

\newenvironment{pf}{\par\noindent{\em Proof}.}{\hfill\framebox(6,6)
\par\medskip}

\newtheorem{theorem}[subsection]{Theorem}
\newtheorem{proposition}[subsection]{Proposition}
\newtheorem{lemma}[subsection]{Lemma}
\newtheorem{corollary}[subsection]{Corollary}

\theoremstyle{definition}
\newtheorem{definition}[subsection]{Definition}

\title[On a question of Colliot-Th\'{e}l\`{e}ne on Chow group mod $n$]{On a question of Colliot-Th\'{e}l\`{e}ne on Chow group mod $n$}
\author{Kalyan Banerjee, Kalyan Chakraborty}

\begin{abstract}
In this note we  define the notion of Tate-Shafarevich group and Selmer group of the Chow group of an abelian variety defined over a number field. In this context we give positive answer to the question of Colliot-Th\'{e}l\`{e}ne  that the Chow group of zero cycles modulo $n$ is finite over any number field.
\end{abstract}
\maketitle

\section{Introduction}
The notion of Selmer group and the Tate-Shfarevich group are important from the perspective of local-global principle in arithmetic geometry. The studies of these groups have been initiated by Cassels, Lang, Selmer, Shafarevich, Tate \cite{CAS1},\cite{CAS2},\cite{LT},\cite{Sel}, \cite{Sh1}, \cite{T}. The famous conjecture about the Tate-Shafarevich group tells that this group associated to an abelian variety is finite.  The first case where it has been proven is the case of elliptic curves with complex multiplication having rank atmost 1, by Karl Rubin, \cite{Ru1}. The next is the case of modular elliptic curves with analytic rank atmost 1, by V.Kolyvagin, \cite{Kol}. The paper by Selmer \cite{Sel} has many examples of genus one curves for which the Tate-Shafarevich group has non-trivial elements. The perception of the Tate-Shafarevich group of abelian varieties comes from the first Galois cohomology of the Abelian variety defined over a number field. Alternatively it can be described as the non-trivial torsors on the abelian variety which become trivial over a local field. On  the other hand, Selmer group has been defined by a certain kernel at the level of first Galois cohomology and it is known that this group is finite.

Our aim of this paper is to study the notion of Selmer group and the Tate-Shafarevich group from the perspective of algebraic cycles. That is we consider the Galois action of the absolute Galois group of a number field on the group of degree zero cycles on an abelian variety defined over a number field. Then consider the Selmer and the Tate-Shafarevich group associated to this group of degree zero cycles on the abelian variety by considering the kernel at the level of first Galois cohomology of this particular Galois module.

There are certain things known about the group of degree zero cycles on a smooth projective variety over the algebraic closure of a number field. One is the Mumford-Roitman, argument \cite{M},\cite{R1} about Chow schemes which says that the natural map from the symmetric power of an abelian variety to the Chow group has fibers given by a countable union of Zariski closed subsets in the symmetric power. This result enables us to give a scheme theoretic structure on the first Galois cohomology of the group of degree zero cycles on the abelian variety. Next, there is the famous theorem due to Roitman, \cite{R2}, which says that the torsion subgroup of the group of degree zero cycles on an abelian variety and the torsion subgroup on the abelian variety are isomorphic. This leads us to study divisibility of the group of the degree zero cycles defined over the number field from a cohomological prespective.

The question whether Chow group modulo a positive integer is finite or not is of importance and it is a question due to Colliot-Th\'{e}l\`{e}ne that whether Chow group modulo $n$ is always finite, for any separably closed field extension $K$ of $\QQ$. This question appeared in the paper by Chad Schoen \cite{Sc}, where he gave an example of a variety on which the Chow group of codimension two cycles modulo $n$ is not finite. In our paper we provide a positive answer to the question of Colliot-Th\'{e}l\`{e}ne when we consider zero cycles on an abelian variety over a number field. This assumption is much weaker than that has been assumed in the paper by Schoen. In this direction, other works are due to Totaro \cite{To} and Saito-Sato \cite{SS}.

We consider a new Galois theoretic invariant namely the Selmer group  of the Chow group of zero cycles on  the abelian variety as discussed above. We prove the finiteness of the Selmer group associated to the group of degree zero cycles on an abelian variety modulo rational equivalence. The main result of the paper is  the following:

\smallskip

\begin{theorem}
The Selmer group associated to the Chow group of degree zero cycles on an abelian variety over a number field is finite.
\end{theorem}
As a corollary we prove the following result:
\begin{corollary}
Let $A$ denote an abelian variety defined over a number field $K$. Let $A_0(A)$ denote the group of degree zero cycles on the abelian variety. Let $A_0(A)(K)$ denote the group of Galois-fixed points under the action of the absolute Galois group $G$. Then the group $A_0(A)(K)/nA_0(A)(K)$ is finite.
\end{corollary}
The importance of this result from the perspective of algebraic cycles lies in the Beilinson's conjecture on  the albanese kernel which says that the kernel of the albanese map from the group of degree zero cycles on a smooth projective variety over $\bar \QQ$ to the albanese variety has trivial kernel. It is known that this restriction on the ground field is sharp. That is if we consider a one variable transcendental extension of the field $\bar \QQ$, then over this field there are varieties for which the albanese kernel is non-trivial (see \cite {GG},\cite{GGP}). From the above result we can conclude that the quotient $T(A(K))/nT(A(K))$ of the Albanese kernel (denoted by $T(A)$) for an abelian variety $A$ is finite. Here $n$-is a positive integer and $A(K)$ denote the group of $K$-points on $A$. So it is worth studying the Tate-Shafarevich group and the Selmer group of the albanese kernel. This is why we need an approach to understand the Tate-Shafarevich group and the Selmer group of the group of degree zero cycles in terms of Tate-Shafarevich group and Selmer group of finite product of Jacobians of smooth projective curves on the abelian variety. This will be dealt in detail in a sequel. Also the results due to Schoen and Totaro, \cite{Sc}, \cite{To} tell us that for higher codimensional cycles the answer to the question of Colliot-Th\'{e}l\`{e}ne is negative. We intend to prove the finiteness of Chow group mod $n$ for higher codimensional cycles under certain conditions or put in another way we would be  interested to understand the obstruction of the finiteness of Chow group mod $n$ in terms of cycle-theoretic phenomena.

{\small \textbf{Acknowledgements:} The first  author thanks the hospitality of Harish Chandra Research Institute India for hosting this project. Both the authors thank Azizul Hoque for carefully listening to the arguments present in the paper.}

\section{Tate-Shafarevich group of the Chow group of an abelian variety}
Let $K$ be a number field and let $\overline{K}$ denote its algebraic closure. Let $A$ be an abelian variety defined over the number field $K$. Then we have a natural Galois action of the absolute Galois group $\Gal(\bar K/K)$ and this  action induces further an action on the Chow group of zero cycles on the abelian variety $A$. Here the Chow group is the free abelian group generated by closed points on $A(\bar K)$ modulo the rational equivalence. We denote this group by $\CH_0(A)$.

Consider continuous functions $f$ from $G=\Gal(\bar K/K)$ to $\CH_0(A)$ satisfying
$$
f(\sigma\tau)=f(\sigma)+\sigma f(\tau)\;.
$$
The set of all such functions form a group denoted by $Z^1(G,\CH_0(A))$. Let us consider the subgroup of $Z^1(G,\CH_0(A))$ consisting of elements $f$ such that
$$
f(\sigma)=\sigma.x-x
$$
where $x\in\CH_0(A)$ and we denote this subgroup by $B^1(G,\CH_0(A))$. Let
$$
H^1(G,\CH_0(A)):=Z^1(G,\CH_0(A))/B^1(G,\CH_0(A)).
$$
There is a natural homomorphism of abelian groups from $\CH_0(A)$ to $A$ and so by functoriality of group cohomology  this homomorphism descends to a homomorphism from $H^1(G,\CH_0(A))$ to $H^1(G,A)$. The map from $\CH_0(A)$ to $A$ is denoted by $\alb$, the albanese map. We denote the map from $H^1(G,\CH_0(A))$ to $H^1(G,A)$ as $\alb$. We intend to understand the structure of  $H^1(G,\CH_0(A))$. To that extent, consider the natural map from $\Sym^n A$ to $\CH_0(A)$, which sends an unordered $n$-tuple $\{P_1,\cdots,P_n\}$ of $\bar K$ points on $A$ to the cycle class
$$
\sum_{i=1}^n [P_i]\;.
$$
Now $H^1(G,\CH_0(A))$ is  isomorphic to the colimit of Galois cohomology of finite groups
$$
H^1(\Gal(L/K),\CH_0(A_L))\;
$$
 with $L/K$ is a finite Galois extension and $A_L$ is the collection of $L$-points on $A$. As $G$ is a profinite, the range of any function $\eta$ from $G$ to $\CH_0(A)$ is finite.  Let $Z_{l,m}$  be the collection of all maps $\eta$ from $G$ to $\CH_0(A)$ such that $\eta$ factors through $\Sym^l A\times \Sym^m A$ (this can be obtained by decomposing a zero cycle into positive and negative parts). That is, we identify the maps $\eta$, factoring through $\Sym^l A\times \Sym^m A$, with its image inside $\Sym^l A\times \Sym^m A$. There exists a normal subgroup of $G$ of finite index (call it $N$), such that $\eta$  factors through $G/N$. On the other hand, suppose there is a collection of points on $\Sym^l A\times \Sym^m A$, then we can define a map from $G/N$ to $\Sym^l A\times \Sym^m A$ by assigning the cosets of $N$ to this finite collection of points of $\Sym^l A\times \Sym^m A$. Such a map will be continuous from $G/N$ to $\Sym^l A\times \Sym^m A$, equipped with discrete topology, as $G/N$ is finite. Since the quotient map from $G$ to $G/N$ is  continuous, the map from $G$ to $\Sym^l A\times \Sym^m A$ is also continuous. But these maps are non-canonical as it depends on the choice of the points and their assignments to the left cosets of $N$.  The relation that defines $Z^1(G,\CH_0(A))$ is,
$$
\eta(\sigma \tau)=\eta(\sigma)+\sigma\eta(\tau)\;.
$$
Since this relation happens on $\CH_0(A)$ we have that the cycles
$$
\eta(\sigma\tau)
$$
is rationally equivalent to
$$
\eta(\sigma)+ \sigma\eta(\tau)\;.
$$
Thus there exists a map from $\PR^1_{\bar K}$ to $\Sym^d A$, and a positive zero cycle $B$ such that
$$
f(0)=\eta(\sigma\tau)+B, \quad f(\infty)=\eta(\sigma)+\sigma\eta(\tau)+B\;.
$$
Now appealing to  the theorem of Roitman \cite{R}, the collection of $\eta$, such that
$$
\eta(\sigma\tau)
$$
{is rationally equivalent to}
 $$
 \eta(\sigma)+\sigma\eta(\tau),
 $$
is a countable union of Zariski closed subsets inside the symmetric power $\Sym^l A\times \Sym^m A$ such that range of $\eta$ is contained in $\Sym^l A\times \Sym^m A$.   We include a proof of this  argument (it can be found in \cite{R}) for the sake of completeness:
\begin{theorem}
The collection of all $\eta$ contained in $Z_{l,m}$ such that
$$
\eta(\sigma\tau)
$$
which is rationally equivalent to
$$
\eta(\sigma)+\sigma\eta(\tau),
$$
is a countable union of Zariski closed subsets inside $\Sym^l A\times \Sym^l A$ (denoted by $Z^1_{l}$).
\end{theorem}
\begin{proof}
Let us consider the following reformulation of the definition of rational equivalence. Let $Z_1$ and $Z_2$ be two codimension $0$-cycles. They are rationally equivalent if there exists a positive $0$-cycle $B$, such that $Z_1+B,Z_2+B$ belong to $\Sym^d(A)$ for some fixed $d$, and there exists a regular morphism $f$ from $\PR^1_{\bar K}$ to $\Sym^d(A)$, such that
$$
f(0)=Z_1+B,\quad f(\infty)=Z_2+B\;.
$$
Let  $\eta$ be such that its range is contained in $\Sym^{l} A\times \Sym^l A$. Then
$$
\eta(\sigma\tau)
$$
{is rationally equivalent to }
$$\eta(\sigma)+\sigma\eta(\tau)\;,$$
for every $\sigma,\tau$.
 So there exists a positive cycle $B$ and a regular map $f$ from $\PR^1_{\bar K}$ to $\Sym^{l+u}(A)$  such that
$$f(0)=\eta(\sigma\tau)+B,\quad f(\infty)=\eta(\sigma)+\sigma\eta(\tau)+B\;.$$
Thus it is natural to consider the  subvarieties of $\prod_{i=1}^n \Sym^{l}(A)\times \Sym^{l}(A)$ denoted by $W^{u,v}_{l,n}$ consisting of $\eta$ such that the above equation is satisfied.

Let $(\eta,f,B)$ be such  that  $B\in \Sym^{u}(A)$ and  $f$ in $\Hom^v(\PR^1,\Sym^{l+u}(A))$, for some positive integer $v$ satisfying
$$
f(0)=\eta(\sigma\tau)+B,\quad f(\infty)=\eta(\sigma)+\sigma\eta(\tau)+B\;,
$$
for fixed $\sigma,\tau$.

For simplicity we denote $\PR^1_{\bar K}$ as $\PR^1$.
Here $\Hom^v(\PR^1,\Sym^{l+u}(A))$ is the Hom scheme of degree $v$ morphisms from $\PR^1$ to $\Sym^{l+u}(A)$.
Let us denote $\prod_{i=1}^n \Sym^{d_i}(A)$ as $\Sym^{d_1,\cdots,d_n}(A)$.

  Let
  $$
  e:\Hom ^v(\PR ^1,\Sym^{l+u,l+u}(A))\to
  \Sym^{l+u,l+u}(A)
  $$
be the evaluation morphism sending $f:\PR ^1\to \Sym^{l+u,l+u}(A)$ to the ordered pair $(f(0),f(\infty ))$, and,
  $$
  s:\prod_{i=1}^n (\Sym^l A\times \Sym^l A)\times {\Sym^{u}(A)}\to \Sym^{l+u,l+u}(A)
  $$
be  the regular morphism sending $(\eta,B)$ to $(\eta(\sigma\tau)+B,\eta(\sigma)+\sigma\eta(\tau))+B)$.  Subsequently $e$ and $s$ allow us to consider the fibred product
  $$
  V=\Hom ^v(\PR ^1,\Sym^{l+u,l+u}(A))\times _{\Sym^{l+u,l+u}(A)}\prod_{i=1}^n (\Sym^l A\times \Sym^l A)\times {\Sym^{u}(A)}\; .
  $$
$V$ is a  union of Zariski closed subsets in the product
  $$
  \Hom ^v(\PR ^1,\Sym^{n+u,n+u}(A))\times \prod_{i=1}^n (\Sym^l A\times \Sym^l A)\times {\Sym^{u}(A)}
  $$
over $\Spec (\bar K)$ consisting of tuples $(f,\eta,B)$ such that
  $$
  e(f)=s(\eta,B)\; .
  $$
That is,
  $$
  (f(0),f(\infty ))=(\eta(\sigma\tau)+B,\eta(\sigma)+\sigma\eta(\tau)+B)\; .
  $$
The latter equality gives
  $$
  \pr_2(V)= W_{l,n}^{u,v}\; .
  $$
Vice versa, if $\eta$ is a closed point of $W_{l,n}^{u,v}$, there exists a regular morphism
  $$
  f\in \Hom ^v(\PR ^1,\Sym^{l+u,l+u}(A))
  $$
with $ f(0)=\eta(\sigma\tau)+B$ and $ f(\infty )=\eta(\sigma)+\sigma\eta(\tau)+B$. Then $(f,\eta,B) \in V$.
So $W_{l,n}^{u,v}$ is a finite  union of quasi-projective varieties (or of constructible sets).

Let $(\eta) \in W_{l,n}^{u,v}$. Then there exists $f \in \Hom^v(\PR^1,\Sym^{l+u}(A))$ and $B\in \Sym^u(A)$ such that
$$
f(0)=\eta(\sigma\tau)+B,\quad f(\infty)=\eta(\sigma)+\sigma\eta(\tau)+B\;.
$$
Hence
 $$
 (\eta(\sigma\tau)+B,\eta(\sigma)+\sigma\eta(\tau)+B)\in W_{l+u,n}^{0,v}\;.
 $$
On the other hand consider the map
$$
\tilde{s}:\prod_{i=1}^n (\Sym^l A\times \Sym^l A)\times \Sym^{u}(A)\to \Sym^{l+u,l+u}(A)
$$
given by
$$
(\eta,B)\mapsto (\eta(\sigma\tau)+B,\eta(\sigma)+\sigma\eta(\tau)+B)\;.
$$
Then by above,
$$
W_{l,n}^{u,v}\subset \pr_{1,2}(\tilde{s}^{-1}(W_{l,n+u}^{0,v}))\;.
$$
Conversely, suppose that $(\eta(\sigma\tau)+B,\eta(\sigma)+\sigma\eta(\tau)+B)\in \pr_{1,2}(\tilde{s}^{-1}(W_{l,n+u}^{0,v}))$. Then there exists $f \in \Hom^v(\PR^1,\Sym^{l+u}(A))$ satisfying
$$
f(0)=\eta(\sigma\tau)+B,\quad f(\infty)=\eta(\sigma)+\sigma\eta(\tau)+B\;.
$$
Which implies that $\eta\in W_{l,n}^{u,v}$ and thus
  $$
  W_{l,n}^{u,v}=\pr _{1,2}(\tilde s^{-1}(W_{l+u,n}^{0,v}))\; .
  $$

Since $\tilde s$ is continuous and $\pr _{1,2}$ is proper,
  $$
  \bar W_{l,n}^{u,v}=
  \pr _{1,2}(\tilde s^{-1}(\bar W_{l+u,n}^{0,v})\; .
  $$
So it is enough to show that $\bar W_{l,n}^{0,v}\subset W_{l,n}$.

Let $\eta$ be a closed point of $\bar W_{l,n}^{0,v}$.  Suppose
  $$
  \eta\in \bar W_{l,n}^{0,v}\smallsetminus W_{l,n}^{0,v}\; .
  $$
Let $W$ be an irreducible component of the countable union of the  quasi-projective varieties $W_{l,n}^{0,v}$ whose Zariski closure $\bar W$ contains the point $\eta$. Let $U$ be an affine neighbourhood of $\eta$ in $\bar W$. Now $U\cap W$ is non-empty  as $\eta$ is in the closure of $W$.

We show that one can always take an irreducible curve $C$ passing through $\eta$ in $U$. Let us consider $U$ as $\Spec (A)$. It is enough to show that there exists a prime ideal in $\Spec (A)$ of height $n-1$, where $n$ is the dimension of $\Spec (A)$, where $A$ is Noetherian. Since $A$ is of dimension $n$ there exists a maximal chain of prime ideals
 $$
  \gop _0 \subset \gop _1 \subset \cdots \subset \gop_n=\gop.
 $$
 Now consider the subchain
 $$
   \gop _0 \subset \gop _1 \subset \cdots \subset \gop _{n-1}\; .
 $$
This is a chain of prime ideals and $\gop _{n-1}$ is a prime ideal of height $n-1$ and thus we get an irreducible curve.

Let $\bar C$ be the Zariski closure of $C$ in $\bar W$. Two evaluation  morphisms $e_0$ and $e_{\infty }$ from $\Hom ^v(\PR ^1,\Sym^{l}(A))$ to $\Sym^{l}(A)$ give the regular morphism
  $$
  e_{0,\infty }:\Hom ^v(\PR ^1,\Sym^{l}(A))\to \Sym^{l,l}(A)\; .
  $$
Then $W_{l,n}^{0,v}$ is exactly the Cartesian product  of $\prod_{i=1}^n \Sym^l A\times \Sym^l A$ and the Hom scheme $\Hom^v(\PR^1,\Sym^{l}A)$ over $\Sym^{l}A\times \Sym^{l}A$. Here the map from $\prod_{i=1}^n \Sym^l A\times \Sym^l A$ to $\Sym^{l,l}A$ is the evaluation map
$$
\eta\mapsto (\eta(\sigma\tau),\eta(\sigma)+\sigma\eta(\tau))
$$
and we can choose a quasi-projective curve $T$ in $\Hom ^v(\PR ^1,\Sym^{l}(A))$, such that the closure of the image $e_{0,\infty }(T)$ is $\bar C$.

For that consider the curve $C$ in $W$ so it is contained in $W_{l,n}^{0,v}$ and this set is given by the above Cartesian product. We consider the inverse image of $\bar C$ under the morphism $e_{0,\infty }$ and as $\bar C$ is a curve,  $\dim(e_{0,\infty }^{-1}(C))\geq 1$. So it contains a curve. Consider two points on $\bar C$ and consider their inverse images under $e_{0,\infty }$. As $\Hom ^v(\PR ^1,\Sym^{l}(X))$ is a quasi  projective variety, $e_{0,\infty }^{-1}(\bar C)$ is also quasi projective, therefore we can embed it into some $\PR ^m$ and consider a smooth hyperplane section through the two points fixed above. Continuing this procedure we get a curve containing these two points and contained in $e_{0,\infty }^{-1}(C)$. Therefore we get a curve $T$ mapping onto $\bar C$ and thus the closure of the image of $T$ is $\bar C$.

Now, as we have noted above, $\Hom ^v(\PR ^1,\Sym^{l}(A))$ is a quasi-projective variety and thus can be embedded into some projective space $\PR ^m$. Let $\bar T$ be the closure of $T$ in $\PR ^m$ and $\tilde T$ be the normalization of $\bar T$.  Also let $\tilde T_0$ be the pre-image of $T$ in $\tilde T$. We consider the composition
  $$
  f_0:\tilde T_0\times \PR ^1\to T\times \PR ^1\subset
  \Hom ^v(\PR ^1,\Sym^{l}(A))\times \PR ^1
  \stackrel{e}{\to }\Sym^{l}(A)\; ,
  $$
where $e$ is the evaluation morphism $e_{\PR ^1,\Sym^{l}(A)}$. The regular morphism $f_0$ defines a rational map
  $$
  f:\tilde T\times \PR ^1\dasharrow \Sym^{l}(A)\;.
  $$
Then by resolution of singularities, $f$ could be extended to a regular map from $(\tilde T\times \PR^1)'$ to $\Sym^{l}(A)$, where $(\tilde T\times \PR^1)'$ denotes the blow up of $\tilde T\times \PR^1$ along the indeterminacy locus which is a finite set of points. We continue to call the strict transform of $\tilde T $ in the blow up as $\tilde T$, and the pre-image of $T$ as $\tilde {T_0}$

Now the regular morphism $\tilde T_0\to T\to \bar C$ extends to the regular morphism $\tilde T\to \bar C$. Let $P$ be a point in the fibre of this morphism at $\eta$. For any closed point $Q$ on $\PR ^1$ the restriction $f|_{\tilde T\times \{ Q\} }$ of the rational map $f$ onto $\tilde T\times \{ Q\} (\simeq \tilde T)$ is regular on the whole curve $\tilde T$, ($\tilde T$ is non-singular). Then
  $$
  (f|_{\tilde T\times \{ 0\} })(P)=\eta(\sigma\tau)
  \qquad \hbox{and}\qquad
  (f|_{\tilde T\times \{ \infty \} })(P)=\eta(\sigma)+\sigma\eta(\tau)\; .
  $$
Now
$$
f:\{P\}\times \PR^1\to \Sym^{l}(A)
$$
has the property that
$$
f(0)=\eta(\sigma\tau),\quad f(\infty)=\eta(\sigma)+\sigma\eta(\tau)\;.
$$
Hence  $W_{l,n}^{0,v}$ is Zariski closed and that in turn gives $W_{l,n}^{u,v}$ is infact Zariski closed.
\end{proof}
Similarly we can prove that the collection of $\eta$ in $Z_{l,m}$ such that $\eta(\sigma)$ is rationally equivalent to $\sigma.z-z$ (for a fixed zero cycle $z$) is a countable union of Zariski closed subsets in $Z_{l,m}$. We call it $B^1_{l,m}$.
Therefore we can conclude from the above theorem that:
\begin{theorem}
$H^1(G,\CH_0(A))$ admits a surjective map from the countable union $\cup_{l,m}Z^1_{l,m}$ such that $\cup_{l,m}B^1_{l,m}$ is mapped to a point under this surjection.
\end{theorem}
We continue exploring this map $Z^1_{l,m}$ to $H^1(G,\CH_0(A))$. Let $\eta$ be an element in the set $Z^1_{l,m}$. Then for every $(\sigma,\tau)$, we have,
$$
\eta(\sigma\tau)=\eta(\sigma)+\sigma\eta(\tau)\;.
$$
This equality happens in $\CH_0(A)$ and thus consider the tuples
$$
(\eta,f,B)\in Z_{l,m}\times \Hom^v(\PR^1,\Sym^{n+u,n+u}A)\times \Sym^u B
$$
such that the following equations are satisfied:
$$
f(0)=\eta(\sigma\tau)+B
$$
$$
f(\infty)=\eta(\sigma)+\sigma\eta(\tau)+B\;.
$$
If we denote the above quasiprojective variety by $\bcV$ and consider the projection map from $\bcV$ to $\Hom^v(\PR^1,\Sym^{n+u,n+u}A)$, then it is a $\PR^1$-bundle, which is the pull-back of the $\PR^1$-bundle given by
$$
\{(x,f)|x\in \im (f)\}\subset \Sym^{n+u,n+u}A\times \Hom^v(\PR^1,\Sym^{n+u,n+u}A)\;.
$$
Thus over $Z^1_{l,m}$ we have the universal variety $\bcU^1_{l,m}$ consisting of tuples $(\eta,f,B)$ such that the above equations are satisfied and it has the structure of a rationally connected fibration over the Hom-scheme. Therefore if we consider the finite map from $A^{l+m}$ to $\Sym^l A\times \Sym^m A$, the degree of this finite map is $l!m!$. The pullback of $Z^1_{l,m}$ under this map is a finite branched cover of $Z^1_{l,m}$ denoted by $\widetilde{Z^{1}_{l,m}}$. In turn we have the pull-back of the universal family $\bcU^1_{l,m}$ over $\widetilde{Z^{1}_{l,m}}$ denoted by $\widetilde{\bcU^{1}_{l,m}}$. This is a family of branched covers of $\PR^1$ over the Hom-scheme.

\subsection{The group cohomology of the group of degree zero cycles on A}
Let $A_0(A)$ denote the group of degree zero cycles or the zero cycles algebraically equivalent to zero on $A$. Then there is a natural homomorphism from $A^n$ to $A_0(A)$ given by
$$\sum_i P_i\mapsto \sum_i [P_i-n0]\;,$$
where $0$ is the neutral element of the abelian variety $A$. Then the map from $A^n$ to $A_0(A)$ induces by functoriality a natural homomorphism from $H^1(G,A^n)$ to $H^1(G,A_0(A))$. If we consider the natural map from $A^n$ to $A^{n+1}$ given by
$$
(P_1,\cdots,P_n)\mapsto (P_1,\cdots,P_n,0),
$$
then this map gives rise to the homomorphism from $H^1(G,A^n)$ to $H^1(G,A^{n+1})$ and the homomorphism
$$\theta_n:H^1(G,A^n)\to H^1(G,A_0(A))$$
factors through the above map
$$
H^1(G,A^n)\mapsto H^1(G,A^{n+1})\;.
$$
Thus there is a natural homomorphism from the colimit of the groups $H^1(G,A^n)$ to $H^1(G,A_0(A))$, denoted by $\theta$. Thus
$$
\theta:\varinjlim H^1(G,A^n)\to H^1(G,A_0(A))\;.
$$
Now for each $n$  the group law from $A^n$ to $A$ is given by
$$
(a_1,\cdots,a_n)\mapsto \sum_i a_i,
$$
and this  map gives rise to a natural map from $H^1(G,A^n)$ to $H^1(G,A)$. Note that this map factors through the homomorphism
$$
H^1(G,A^n)\to H^1(G,A^{n+1})\;.
$$
Thus there is a homomorphism
$$
\varinjlim H^1(G,A^n)\to H^1(G,A)\;.
$$
Since the map $H^1(G,A^n)$ to $H^1(G,A_0(A))$ factors through
$$
H^1(G,A)\to H^1(G,A_0(A)),
$$
we have that the map
$$
\varinjlim H^1(G,A^n)\to H^1(G,A_0(A))
$$
factors through the map
$$
H^1(G,A)\to H^1(G,A_0(A))\;.
$$
Now the group on the left is the  Weil-Chatelet group of the respective $A$, which consists of the equivalence classes of principal homogeneous spaces over $A$. This group is denoted by $WC(A)$. Using the identification
$H^1(G,A)\cong WC(A)$,
we have that,
$$
\theta: WC(A)\to H^1(G,A_0(A))\;.
$$
It is natural to consider when this map is injective and surjective.

Now due to the famous result on torsion subgroup in $A_0(A)$ \cite{R2}, we know that this group of torsion is isomorphic to the group of torsion in $A$. So we expect a similar result when we consider the group cohomology
$H^1(G,A)$ and $H^1(G,A_0(A))$.
\begin{theorem}
The kernel of the map $H^1(G,A)[n]\to H^1(G,A_0(A))[n]$ is isomorphic to the group
$$
A_0(A)(K)/nA_0(A)(K)\;.
$$
\end{theorem}
\begin{proof}
We consider the exact sequence of abelian groups:
$$
0\to A[n]\to A\stackrel{n}{\to} A\to 0
$$
where $A[n]$ is the group of $n$-torsion on $A$. The map $A\stackrel{n}{\to}A$ is
$$
a\mapsto na\;.
$$
Corresponding to this exact sequence we have the  long exact sequence :
$$
0\to A[n](K)\to A(K)\to A(K)\to H^1(G,A[n])\to H^1(G,A)\to H^1(G,A)
$$
which gives the exact sequence
$$
0\to A(K)/nA(K)\to H^1(G,A[n])\to H^1(G,A)[n]\to 0
$$
Here $H^1(G,A)[n]$ is the group of $n$-torsion of $H^1(G,A)$.
Similarly we have the short exacts sequence
$$
0\to A_0(A)(K)/nA_0(A)(K)\to H^1(G,A_0(A)[n])\to H^1(G,A_0(A))[n]\to 0\;.
$$
Now we consider the natural maps
$$
H^1(G,A[n])\to H^1(G,A_0(A)[n])
$$
and
$$H^1(G,A)[n]\to H^1(G,A_0(A))[n]\;.$$
The first of the above is an isomorphism because
$$
A[n]\cong A_0(A)[n]\;.
$$
So consider the map
$$
\theta:H^1(G,A)[n]\to H^1(G,A_0(A))[n]
$$
and let $\theta(a)=0\;.$
Let us denote the maps
$$
H^1(G,A[n])\to H^1(G,A)[n]
$$
by $\phi_A$ and
$$
H^1(G,A_0(A)[n])\to H^1(G,A_0(A))[n]
$$
by $\phi_{A_0(A)}$. Thus from the above
$$
\theta\phi_A(b)=0
$$
where $\phi_A(b)=a\;.$
Therefore it follows that
$$
\phi_{A_0(A)}\theta(b)=0
$$
Now since $\theta $ is an isomorphism from $H^1(G,A[n])$ to $H^1(G,A_0(A)[n])$, there exists a unique $c$ such that
$$
\theta^{-1}(c)=b\;.$$
Hence
$$
\phi_{A_0(A)}(c)=0\;
$$
and therefore $c\in A_0(A)(K)/nA_0(A)(K)$.

Conversely, for $c\in A_0(A)(K)/nA_0(A)(K)$,
$$
\phi_{A_0(A)}(c)=0\;.
$$
Since there exists unique $b$ such that $\theta(b)=c$,  we have,
$$
\phi_{A_0(A)}\theta(b)=0=\theta\phi_A(b)=0\;.
$$
So the kernel of $\theta$ from $H^1(G,A)[n]$ to $H^1(G,A_0(A))[n]$ is in bijection with the group $A_0(A)(K)/nA_0(A)(K)\;.$ This later group is the images of the increasing union of the groups $A^m(K)/nA^m(K)$. These groups are all finite hence we have that
$$
A_0(A)(K)/nA_0(A)(K)
$$
is a pro-finite group.
\end{proof}
\section{Tate-Shafarevich and Selmer group of $A_0(A)$ and their properties}
We have  the exact sequence
$$
0\to A_0(A)(K)/nA_0(A)(K)\to H^1(G,A_0(A)[n])\to H^1(G,A_0(A))[n]\to 0\;.
$$
Let  $v$ be a place of $K$ and $K_v$ be  the corresponding completion.
 Then consider the algebraic closure $\bar K_v$ of $K_v$ and embed $\bar K$ into $\bar K_v$. This embedding gives an injection of the Galois group $\Gal(\bar K_v/K_v)$ into $\Gal(\bar K/K)$. Considering the Galois cohomology, we have a homomorphism,
$$
H^1(\Gal(\bar K/K),A_0(A(\bar K))\to H^1(\Gal(\bar K_v/K_v),A_0(A(\bar K_v)))\;.
$$
In short we write the groups $\Gal(\bar K/K)$ and $\Gal(K_v/K_v)$ as $G$ and $G_v$ respectively. We have the following commutative diagrams:

$$
  \diagram
  A_0(A)(K)/nA_0(A)(K)\ar[dd]_-{} \ar[rr]^-{} & & H^1(G,A_0(A)[n]) \ar[dd]^-{} \\ \\
  A_0(A)(K_v)/nA_0(A)(K_v) \ar[rr]^-{} & & H^1(G_v,A_0(A_v)[n])
  \enddiagram
  $$

$$
  \diagram
  H^1(G,A_0(A)[n])\ar[dd]_-{} \ar[rr]^-{} & & H^1(G,A_0(A))[n] \ar[dd]^-{} \\ \\
 H^1(G_v,A_0(A_v)[n])\ar[rr]^-{} & & H^1(G_v,A_0(A_v))[n]
  \enddiagram
  $$
Let us consider the map
$$
H^1(G,A_0(A)[n])\to \prod_v H^1(G_v,A_0(A_v)).
$$
\begin{definition}
The kernel of this map is defined to be the Selmer group associated to the map $z\mapsto nz$ and it is denoted by $S^n(A_0(A)/K)$.
\end{definition}

\begin{definition}
The Tate-Shafarevich group is the kernel of the map
$$H^1(G,A_0(A))\to \prod_v H^1(G_v,A_0(A_v))\;,$$
and it is denoted by $TS(A_0(A)/K)$.
\end{definition}
Further we consider the commutative diagram:
$$
  \diagram
  H^1(G,A[n])\ar[dd]_-{} \ar[rr]^-{} & & \prod_v H^1(G_v,A_v)[n] \ar[dd]^-{} \\ \\
 H^1(G,A_0(A)[n])\ar[rr]^-{} & & \prod_v H^1(G_v,A_0(A_v))[n]
  \enddiagram
$$
By Roitman's theorem \cite{R2}, the groups $A[n]$ and $A_0(A)[n]$ are isomorphic, therefore the group cohomologies are isomorphic. So the left vertical arrow in the above diagram is an isomorphism. Suppose there is an element in $S^n(A/K)$, then by the commutativity of the above diagram, the image of this element under the left vertical homomorphism is in $S^n(A_0(A)/K)$. We prove the following theorem:
\begin{theorem}
The group $S^n(A_0(A)/K)$ is finite and hence $A_0(A)(K)/nA_0(A)(K)$ is finite.
\end{theorem}
The following lemma
\cite{Sil}[lemma 4.3, chapter X]will be used  in the proof of this theorem.
\begin{lemma}
Let $M$ be a finite $G$ module and $S$ be a set of finitely many places in $K$. Consider
 $$H^1(G,M;S)$$
consisting of all elements $\eta$ in $H^1(G,M)$, which are unramified outside $S$. Then $H^1(G,M;S)$ is finite.
\end{lemma}
\begin{proof} (Proof of the theorem)
Consider the diagram
$$
  \diagram
  H^1(G,A[n])\ar[dd]_-{} \ar[rr]^-{} & & \prod_v H^1(G_v,A_v)[n] \ar[dd]^-{} \\ \\
 H^1(G,A_0(A)[n])\ar[rr]^-{} & & \prod_v H^1(G_v,A_0(A_v))[n]
  \enddiagram
  $$
Suppose
$S^n(A_0(A)/K) \neq \{0\}$. Then by the left vertical isomorphism any such non-trivial element lifts to a unique element in $S^n(A/K)$. The commutativity of the above diagram implies that the image of this lift under the top horizontal arrow is in the kernel of
$$
\prod_v H^1(G_v,A_v)[n]\to \prod_v H^1(G_v,A_0(A_v))[n]\;.
$$
Thus we concentrate on the kernel of the above homomorphism. For that we have the exact sequences:
$$
0\to A(K_v)/nA(K_v)\to H^1(G_v,A_v[n])\to H^1(G_v,A_v)[n]\to 0
$$
and
$$
0\to A_0(A)(K_v)/nA_0(A)(K_v)\to H^1(G_v,A_0(A_v)[n])\to H^1(G_v,A_0(A_v))[n]\to 0
$$
also we have the commutative squares:

$$
  \diagram
  A(K_v)/nA(K_v)\ar[dd]_-{} \ar[rr]^-{} & & H^1(G,A_v[n]) \ar[dd]^-{} \\ \\
  A_0(A)(K_v)/nA_0(A)(K_v) \ar[rr]^-{} & & H^1(G_v,A_0(A_v)[n])
  \enddiagram
  $$

$$
  \diagram
  H^1(G_v,A_v[n])\ar[dd]_-{} \ar[rr]^-{} & & H^1(G_v,A_v)[n] \ar[dd]^-{} \\ \\
 H^1(G_v,A_0(A_v)[n])\ar[rr]^-{} & & H^1(G_v,A_0(A_v))[n]
  \enddiagram
  $$
Suppose  there exists an element $\eta$ such that
 $$
 \eta\in \ker(H^1(G_v,A_v)[n]\to H^1(G_v,A_0(A_v))[n]).
 $$
 Then it means that there exists an element $z\in A_0(A_v)(K_v)$ such that
 $$
 \sigma(\eta)=\sigma.z-z
 $$
 for all $\sigma$ in $G_v$. In particular for all $\sigma$ in the inertia group $I_v$. Now suppose that $v$ is a finite place such that $v\nmid n$ and $A/K$ has good reduction at $v$. Then consider the specialization homomorphism from $A_0(A_v)$ to $A_0(A_v')$, where $A_v'$ is the reduction of $A_v$ at $v$. Then it follows that the image of
 $$\sigma.z-z$$
 is zero  for all $\sigma$ in $I_v$. On the other hand, by the exactness of the above sequence $\sigma.z-z$ is an $n$-torsion, thus by the Roitman's theorem on torsion, the element
 $$
 \sigma.z-z
 $$
corresponds to an $n$-torsion on $A_v$. By similar argument as above, we see that this n-torsion on $A_v$ is mapped to zero under $A_v\to A_v'$. But we know that the $n$-torsion of $A_v$ are embedded in $A_v'$. Therefore we have this $n$-torsion on $A$ is zero and consequently
 $$
 \sigma.z=z
 $$
 for all $\sigma\in I_v$.  Now $\sigma(\eta)$ is in $H^1(G_v,A_0(A_v)[n])$ and it is zero in $H^1(I_v,A_0(A_v)[n])$. Therefore it is zero in $H^1(I_v,A_v[n])$.  Then by the inflation-restriction sequence for group cohomology, the kernel $H^1(G_v,A_v)[n]\to H^1(G_v,A_0(A_v))[n]$ is trivial for all but a finite set of places $v$.  Now consider the image of $H^1(G,A[n])$ in $H^1(G_v,A_v)[n]$. By the previous argument, this image is unramified for all but finitely many places $v$.
 Hence   lemma 3.4
  gives that it is finite.

 Now the image of $H^1(G,A[n])$ in $H^1(G_v,A_v)[n]$ is zero for all but finitely many $v$. So we have obtained that the group $S^n(A_0(A)/K)$ embeds in the finite group, the image of $H^1(G,A[n])$ in
 $$
 \ker(\prod_v H^1(G_v,A_v)[n]\to H^1(G_v,A_0(A_v))[n])\;,
 $$
which  actually is in a finite product
 $$
 \ker(\prod_{v\in S}H^1(G_v,A_v)[n]\to H^1(G_v,A_0(A_v))[n])\;.
 $$
Therefore  $S^n(A_0(A)/K)$ is finite and
hence it follows that
$$
A_0(A)(K)/nA_0(A)(K)
$$ is finite.
\end{proof}
Now we look at the Tate-Shafarevich group.
\begin{theorem}
 $TS(A_0(A)/K)$ admits a map to the direct sum of
$$
TS((\oplus_{i=1}^n J(C_i)/A_n)/K)
$$
where $C_i$'s are finitely many smooth, projective curves in $A$ and $A_n$ is the kernel of the map $\oplus_{i=1}^n J(C_i)\to A$.
\end{theorem}
\begin{proof}
The group $A_0(A(\bar K))$ consists of algebraically trivial zero cycles on $A(\bar K)$. Thus given any such cycle $z$ there exists a smooth projective curve $C$ defined over $\bar K$ and two $\bar K$ points $P,Q$ on the curve such that
$$
j_*(P-Q)=z\;,
$$
where $j$ is the closed embedding of $C$ in $A$. Therefore we have the following commutative diagram:

$$
  \diagram
  \oplus_{C\subset A} J(C)\ar[dd]_-{} \ar[rr]^-{} & & A_0(A) \ar[dd]^-{} \\ \\
 \oplus _{C_v} J(C_v)\ar[rr]^-{} & & A_0(A_v)
  \enddiagram
  $$
Here the horizontal arrows are surjective.

Since the group cohomology of a direct sum is the direct sum of group cohomologies we have
$$H^1(G,\oplus J(C))\cong \oplus H^1(G,J(C))$$
and similarly
$$H^1(G_v,\oplus J(C_v))\cong \oplus H^1(G_v,J(C_v))\;.$$
So we get the following commutative diagram by the functoriality of group cohomology
$$
  \diagram
  \oplus_{C\subset A} H^1(G,J(C))\ar[dd]_-{} \ar[rr]^-{} & & H^1(G, A_0(A)) \ar[dd]^-{} \\ \\
 \oplus _{C_v} H^1(G_v,J(C_v))\ar[rr]^-{} & & H^1(G_v,A_0(A_v))
  \enddiagram
  $$
Now suppose  there is a function $\eta:G\to A_0(A)$ which has finite image (as $G$ is profinite).
So  the images $\eta(\sigma_1),\cdots,\eta(\sigma_n)$ are supported on finitely many $J(C_i)$, where $C_i$ is a smooth projective curve. Then
$$
\eta(\sigma\tau)=\eta(\sigma)+\sigma\eta(\tau)
$$
in $A_0(A)$. Hence the above equality happens on the image of $\oplus_{i=1}^n J(C_i)$ inside $A_0(A)$. Therefore we have
$$
\eta(\sigma\tau)=\eta(\sigma)+\sigma\eta(\tau)
$$
on
$$
\oplus_{i=1}^n J(C_i)/ \ker(\oplus_i j_{i*})
$$
where $j_i$ is the closed embedding of $J(C_i)$ into $A$.

Now the kernel of $\oplus_{i=1}^n j_{i*}$ is contained in the kernel of $\oplus _{i=1}^n J(C_i)\to A$. We call  this kernel as $A_{n}$.  Then if we consider the abelian variety $(\oplus_{i=1}^n J(C_i))/A_n$, we have
$$\oplus_{i=1}^nJ(C_i)\to A_0(A)$$
is surjective onto its image  with the kernel contained in $ A_n$ and similar thing happens for $J(C_{vi})$.  Therefore we have the following commutative diagram:
$$
  \diagram
   \oplus H^1(G,(\oplus_{i=1}^n J(C_i))/\ker(\oplus_i j_{i*}))\ar[dd]_-{} \ar[rr]^-{} & & H^1(G, A_0(A)) \ar[dd]^-{} \\ \\
 \oplus H^1(G_v,(\oplus _{i=1}^n J(C_{vi})/\ker(\oplus_i j_{iv*})))\ar[rr]^-{} & & H^1(G_v,A_0(A_v))
  \enddiagram
  $$
The direct sum on the extreme left of the diagram is taken on finite direct sums $\oplus_{i=1}^n J(C_i)/\ker(\oplus_i j_{i*})$. By the previous discussion the group $\oplus_{i=1}^n J(C_i)/\ker(\oplus_i j_{i*})$ admits a homomorphism to  $(\oplus_{i=1}^n J(C_i))/A_n$. Similarly $(\oplus_{i=1}^n J(C_{vi}))/\ker(\oplus_{i=1}^n j_{vi*})$ admits a homomorphism to $(\oplus_{i=1}^n J(C_{vi}))/A_{nv}$. We thus have lifted the function from $G$ to $A_0(A)$ to $G\to B_n=\oplus_{i=1}^n J(C_i)/\ker(\oplus_i j_{i*})$. So  the map
$$
\oplus H^1(G,B_n)\to H^1(G,A_0(A))
$$
is surjective and in fact is an isomorphism (the direct sum is taken on the quotients of finite direct sums: $\oplus_{i=1}^n J(C_i)/\ker(\oplus_i j_{i*})$). Now if there is an  element in the kernel of
$$
H^1(G,A_0(A))\to H^1(G_v,A_0(A_v))
$$
then by the following commutative diagram:
$$
  \diagram
  \oplus H^1(G,B_n)\ar[dd]_-{} \ar[rr]^-{} & & H^1(G, A_0(A)) \ar[dd]^-{} \\ \\
 \oplus  H^1(G_v,B_{vn})\ar[rr]^-{} & & H^1(G_v,A_0(A_v))
  \enddiagram
  $$
it lifts uniquely to an element in the kernel of
$$
\oplus H^1(G,B_n)\to H^1(G_v,A_0(A_v))\;.
$$
Since
$$
\oplus B_{vn}\to A_0(A_v)
$$
is an isomorphism, the image of the lift under the map
$$
\oplus_i H^1(G,B_n)\to \oplus_i H^1(G_v,B_{vn})
$$
is zero. In fact it is supported on one finite direct sum $\oplus_{i=1}^n J(C_i)/\ker(\oplus_i j_{i*})$. Now we look at  the diagram:
$$
  \diagram
   H^1(G,B_n)\ar[dd]_-{} \ar[rr]^-{} & & H^1(G, \oplus_{i=1}^n J(C_i)/A_n) \ar[dd]^-{} \\ \\
 \oplus  H^1(G_v,B_{vn})\ar[rr]^-{} & & H^1(G_v,\oplus_{i=1}^n J(C_{vi})/A_{nv})
  \enddiagram
  $$
By the previous discussion, the image of the element under the map
$$H^1(G,B_n)\to H^1(G,\oplus_{i=1}^n J(C_i)/A_n)$$
is  in the kernel of
$$H^1(G,\oplus_{i=1}^n J(C_i)/A_n)\to H^1(G_v,\oplus_{i=1}^n J(C_{vi})/A_{nv})$$
which is the Tate-Shafarevich group of the abelian variety $\oplus_{i=1}^n J(C_i)/A_n$. So the Tate-Shafarevich group $TS(A_0(A)/K)$ admits a map to
$$\oplus TS(\oplus_{i=1}^n J(C_i)/A_n)\;.$$
Thus the kernel of this map is contained in $H^1(G,A_n/B_n)$.

\end{proof}

\end{document}